\documentclass[reqno]{amsart}

\usepackage[latin1]{inputenc}
\usepackage{amssymb}
\usepackage{graphicx}
\usepackage{amscd}
\usepackage[hidelinks]{hyperref}
\usepackage{color}
\usepackage{float}
\usepackage{graphics,amsmath,amssymb}
\usepackage{amsthm}
\usepackage{amsfonts}
\usepackage{latexsym}
\usepackage{epsf}
\usepackage{enumerate}
\usepackage{xifthen}
\usepackage{mathrsfs}
\usepackage{dsfont}
\usepackage{makecell}
\usepackage[FIGTOPCAP]{subfigure}
\usepackage{amsmath}
\allowdisplaybreaks[4]
\usepackage{listings}
\usepackage{etoolbox}
\usepackage{fancyhdr}
\usepackage{pdflscape}
\usepackage[title,toc,titletoc]{appendix}
\usepackage{enumitem}

 \newtheoremstyle{mytheorem}
 {3pt}
 {3pt}
 {\slshape}
 {}
 {\bfseries}
 {.}
 { }
 {}

\numberwithin{equation}{section}

\theoremstyle{theorem}
\newtheorem{theorem}{Theorem}[section]

\newtheorem{lemma}[theorem]{Lemma}

\theoremstyle{definition}

\newtheorem{remark}{Remark}[section]

\newcommand{\Keywords}[1]{\ifthenelse{\isempty{#1}}{}{\smallskip \smallskip \noindent \textbf{Keywords}. #1}}
\newcommand{\MSC}[2][2010]{\ifthenelse{\isempty{#2}}{}{\smallskip \smallskip \noindent \textbf{#1MSC}. #2}}
\newcommand{\abstractnote}[1]{\ifthenelse{\isempty{#1}}{}{\smallskip \smallskip \noindent \textsuperscript{\dag}#1}}

\makeatletter
\def\specialsection{\@startsection{section}{1}%
  \z@{\linespacing\@plus\linespacing}{.5\linespacing}%
  {\normalfont}}
\def\section{\@startsection{section}{1}%
  \z@{.7\linespacing\@plus\linespacing}{.5\linespacing}%
  {\normalfont\scshape}}
\patchcmd{\@settitle}{\uppercasenonmath\@title}{\Large\boldmath}{}{}
\patchcmd{\@settitle}{\begin{center}}{\begin{flushleft}}{}{}
\patchcmd{\@settitle}{\end{center}}{\end{flushleft}}{}{}
\patchcmd{\@setauthors}{\MakeUppercase}{\normalsize}{}{}
\patchcmd{\@setauthors}{\centering}{\raggedright}{}{}
\patchcmd{\section}{\scshape}{\large\bfseries\boldmath}{}{}
\patchcmd{\subsection}{\bfseries}{\bfseries\boldmath}{}{}
\renewcommand{\@secnumfont}{\bfseries}
\patchcmd{\@startsection}{\@afterindenttrue}{\@afterindentfalse}{}{}
\patchcmd{\abstract}{\leftmargin3pc}{\leftmargin1pc}{}{}

\def\maketitle{\par
  \@topnum\z@ 
  \@setcopyright
  \thispagestyle{empty}
  \ifx\@empty\shortauthors \let\shortauthors\shorttitle
  \else \andify\shortauthors
  \fi
  \@maketitle@hook
  \begingroup
  \@maketitle
  \toks@\@xp{\shortauthors}\@temptokena\@xp{\shorttitle}%
  \toks4{\def\\{ \ignorespaces}}
  \edef\@tempa{%
    \@nx\markboth{\the\toks4
      \@nx\MakeUppercase{\the\toks@}}{\the\@temptokena}}%
  \@tempa
  \endgroup
  \c@footnote\z@
  \@cleartopmattertags
}
\makeatother

\newcommand{\qbinom}[2]{\begin{bmatrix}#1\\#2\end{bmatrix}}
\newcommand{\fl}[1]{\left\lfloor#1\right\rfloor}

\title[Truncated Gauss' square exponent theorem]{Note on the truncated generalizations of Gauss' square exponent theorem}

\author[S. Chern]{Shane Chern}
\address{Department of Mathematics, The Pennsylvania State University, University Park, PA 16802, USA}
\email{shanechern@psu.edu}

\date{}

\begin{document}

%

\maketitle

\begin{abstract}
In this note, we investigate J.-C. Liu's work on truncated Gauss' square exponent theorem and obtain more truncations. We also discuss some possible multiple summation extensions of Liu's results.


\MSC{11B65, 33D15.}
\end{abstract}

\section{Introduction}

One major topic of $q$-series deals with various $q$-identities, most of which can be treated as the $q$-analog of combinatorial identities. Some celebrated examples include Euler's pentagonal number theorem \cite[Corollary 1.7]{And1976}
\begin{equation}\label{eq:eu}
\prod_{n\ge 1}(1-q^n)=\sum_{k=-\infty}^\infty (-1)^k q^{k(3k+1)/2}
\end{equation}
and Gauss' square exponent theorem \cite[Corollary 2.10]{And1976}
\begin{equation}\label{eq:ga}
\prod_{n\ge 1}\frac{1-q^n}{1+q^n}=\sum_{k=-\infty}^\infty (-1)^k q^{k^2}.
\end{equation}

Interestingly, some $q$-identities involving infinite sums and/or products also have the corresponding truncated version. Before presenting such truncations, we introduce some standard $q$-series notation:
\begin{align*}
(a;q)_n:=&\prod_{k= 0}^{n-1} (1-aq^k),\\
(a;q)_\infty:=&\prod_{k\ge 0} (1-aq^k).
\end{align*}
We also adopt the $q$-binomial coefficient
$$\qbinom{n}{m}=\qbinom{n}{m}_q:=\begin{cases}
\frac{(q;q)_n}{(q;q)_m (q;q)_{n-m}} & \text{if $0\le m\le n$},\\
0 & \text{otherwise}.
\end{cases}$$

In \cite{BG2002}, Berkovich and Garvan combinatorially proved the following finite $q$-identity
\begin{equation}\label{eq:BG}
\sum_{k=-L}^L (-1)^k q^{k(3k+1)/2} \qbinom{2L-k}{L+k} =1.
\end{equation}
If we let $L\to \infty$, then \eqref{eq:BG} becomes \eqref{eq:eu}. In fact, \eqref{eq:BG} is a direct consequence of
\begin{equation}\label{eq:BG-1}
\sum_{r=0}^{\fl{n/2}}(-1)^r q^{\binom{r}{2}} \qbinom{n-r}{r}=\begin{cases}
(-1)^{\fl{n/3}} q^{n(n-1)/6} & \text{if $n\not\equiv 2 \pmod{3}$},\\
0 & \text{if $n\equiv 2 \pmod{3}$},
\end{cases}
\end{equation}
which first appears in \cite{EZ1996}. To see this, we only need to replace $n$ by $3L$, $r$ by $L+k$ and $q$ by $1/q$ in \eqref{eq:BG-1}. On the other hand, Warnaar \cite{War2004} observed that if one  replaces $n$ by $3L+1$, $r$ by $L+k$ and $q$ by $1/q$ in \eqref{eq:BG-1}, another truncated  generalization of Euler's pentagonal number theorem can be derived
\begin{equation}\label{eq:Wa}
\sum_{k=-L}^L (-1)^k q^{k(3k-1)/2} \qbinom{2L-k+1}{L+k} =1.
\end{equation}
Recently, Liu \cite{Liu2017a} obtained more truncated versions of \eqref{eq:eu} with a surprisingly elementary proof.

The truncations of Gauss' square exponent theorem \eqref{eq:ga}, however, mainly come from a different direction. For example, in \cite{GZ2013}, Guo and Zeng showed that for $L\ge 1$
\begin{equation}\label{eq:GZ}
\frac{(-q;q)_\infty}{(q;q)_\infty}\sum_{k=-L}^L (-1)^k q^{k^2} = 1+ (-1)^L \sum_{n=L+1}^\infty \frac{q^{(L+1)n}(-q;q)_L (-1;q)_{n-L}}{(q;q)_n}\qbinom{n-1}{L}.
\end{equation}
The origin of this type of truncations comes from Andrews and Merca's work \cite{AM2012} on Euler's pentagonal number theorem. For other similar truncated theta series, the interested readers may refer to \cite{AM2018,CHM2016,HJZ2016,Kol2015,Mao2015,Yee2015}. Nonetheless, one should admit that \eqref{eq:GZ} is complicated especially comparing with \eqref{eq:BG} and \eqref{eq:Wa}. Hence we would expect truncated  generalization of Gauss' square exponent theorem as neat as \eqref{eq:BG} and \eqref{eq:Wa}. In \cite{Liu2017b}, Liu provided such truncations
\begin{align}
\sum_{k=-L}^L (-1)^k q^{k^2} (-q;q)_{L-k}\qbinom{3L-k+1}{L+k}&=1,\label{eq:Liu1}\\
\sum_{k=-L}^L (-1)^k q^{k^2} (-q;q)_{L-k}\qbinom{3L-k}{L+k}\frac{1-q^{2L}}{1-q^{3L-k}}&=1,\label{eq:Liu2}\\
\sum_{k=-L}^L (-1)^k q^{k^2} (-q;q)_{L-k}\qbinom{3L-k-1}{L+k-1}&=1.\label{eq:Liu3}
\end{align}
All the three identities are respectively direct consequences of identities analogous to \eqref{eq:BG-1}
\begin{align}
\sum_{r=0}^{n}(-1)^r q^{\binom{r}{2}} (-q;q)_{n-r} \qbinom{2n-r+1}{r}&=\begin{cases}
0 & \text{if $n=2m-1$},\\
(-1)^m q^{m(3m+1)} & \text{if $n=2m$},
\end{cases}\label{eq:Liu1-1}\\
\sum_{r=0}^{n}(-1)^r q^{\binom{r}{2}} (-q;q)_{n-r} \qbinom{2n-r}{r}\frac{1-q^n}{1-q^{2n-r}}&=\begin{cases}
0 & \text{if $n=2m-1$},\\
(-1)^m q^{m(3m-1)} & \text{if $n=2m$},
\end{cases}\\
\sum_{r=0}^{n}(-1)^r q^{\binom{r}{2}} (-q;q)_{n-r} \qbinom{2n-r}{r}&=\begin{cases}
(-1)^{m-1}q^{3m^2-3m+1} & \text{if $n=2m-1$},\\
(-1)^m q^{m(3m-1)} & \text{if $n=2m$}.
\end{cases}\label{eq:Liu3-1}
\end{align}
For example, \eqref{eq:Liu1} is deduced by replacing $n$ by $2L$, $r$ by $L+k$ and $q$ by $1/q$ in \eqref{eq:Liu1-1}. We remark that \eqref{eq:Liu1-1} appears in an early paper of Jouhet \cite[Eq.~(2.7)]{Jou2010}.

We have two purposes in this note. The first purpose is to further investigate Liu's results. We then discuss some possible multiple summation extensions of \eqref{eq:Liu1-1} and \eqref{eq:Liu3-1}, whose idea originates from \cite{GZ2008}.

\section{Further investigation of Liu's results}


We start with the following identity deduced from \eqref{eq:Liu1-1}.

\begin{theorem}
For $n\ge 1$,
\begin{align}
\sum_{r=0}^{n}(-1)^r q^{\binom{r}{2}} (-q;q)_{n-r} &\qbinom{2n-r+1}{r}\frac{1-q^{2n+1}}{1-q^{2n-r+1}}\notag\\
&=\begin{cases}
(-1)^m q^{m(3m-1)} & \text{if $n=2m-1$},\\
(-1)^m q^{m(3m+1)} & \text{if $n=2m$}.
\end{cases}\label{eq:th1-1}
\end{align}
\end{theorem}

\begin{proof}
Following Liu's notation, we write the left-hand side of \eqref{eq:Liu1-1} as $U_n$. Namely,
$$U_n=\sum_{r=0}^{n}(-1)^r q^{\binom{r}{2}} (-q;q)_{n-r} \qbinom{2n-r+1}{r}.$$
Then
\begin{align*}
&\sum_{r=0}^{n}(-1)^r q^{\binom{r}{2}} (-q;q)_{n-r} \qbinom{2n-r+1}{r}\frac{1-q^{2n+1}}{1-q^{2n-r+1}}\\
&\quad =\sum_{r=0}^{n}(-1)^r q^{\binom{r}{2}} (-q;q)_{n-r} \qbinom{2n-r+1}{r}\left(1+q^{2n-r+1}\frac{1-q^{r}}{1-q^{2n-r+1}}\right)\\
&\quad = U_n - q^{2n} \sum_{r=1}^n (-1)^{r-1}q^{\binom{r-1}{2}}(-q;q)_{n-r}\qbinom{2n-r}{r-1}\\
&\quad = U_n - q^{2n} \sum_{r=0}^{n-1} (-1)^{r}q^{\binom{r}{2}}(-q;q)_{n-r-1}\qbinom{2n-r-1}{r}\\
&\quad = U_n-q^{2n} U_{n-1}.
\end{align*}
The desired result follows from \eqref{eq:Liu1-1}.
\end{proof}

If we replace $n$ by $2L-1$ and $r$ by $L+k$ in \eqref{eq:th1-1}, then
$$\sum_{k=-L}^{L-1} (-1)^{L+k} q^{\binom{L+k}{2}} (-q;q)_{L-k-1}\qbinom{3L-k-1}{L+k}\frac{1-q^{4L-1}}{1-q^{3L-k-1}}=(-1)^L q^{L(3L-1)}.$$
We then replace $q$ by $1/q$ and notice that
\begin{align*}
\qbinom{n}{m}_{q^{-1}}&=q^{m(m-n)}\qbinom{n}{m}_q
\end{align*}
and
\begin{align*}
(-q^{-1};q^{-1})_n&=q^{-\binom{n+1}{2}}(-q;q)_n.
\end{align*}
Hence
\begin{align*}
(-1)^L q^{-L(3L-1)}=\sum_{k=-L}^{L-1} & (-1)^{L+k} q^{-\binom{L+k}{2}} q^{-\binom{L-k}{2}}(-q;q)_{L-k}\\
&\times q^{(L+k)(-2L+2k+1)}\qbinom{3L-k+1}{L+k}\frac{1-q^{-(4L-1)}}{1-q^{-(3L-k-1)}}.
\end{align*}
This leads to a new truncation of Gauss' square exponent theorem
\begin{theorem}
For $L\ge 1$,
\begin{equation}
\sum_{k=-L}^{L-1} (-1)^k q^{k^2} (-q;q)_{L-k-1}\qbinom{3L-k-1}{L+k}\frac{1-q^{4L-1}}{1-q^{3L-k-1}}=1.
\end{equation}
\end{theorem}

We next observe that
\begin{align}
\qbinom{2n-r+1}{r}\frac{1-q^{2n+1}}{1-q^{2n-r+1}}&=\qbinom{2n-r}{r-1} \frac{1-q^{2n+1}}{1-q^r}\notag\\
&=\qbinom{2n-r}{r-1} \left(1+\frac{q^r(1-q^{2n-r+1})}{1-q^r}\right)\notag\\
&=\qbinom{2n-r}{r-1}+q^r \qbinom{2n-r+1}{r}.\label{eq:pas-var}
\end{align}

On the one hand, we have

\begin{theorem}
For $n\ge 1$,
\begin{align}
\sum_{r=0}^{n}(-1)^r q^{\binom{r-1}{2}} (-q;q)_{n-r} &\qbinom{2n-r+1}{r}\frac{1-q^{2n+1}}{1-q^{2n-r+1}}\notag\\
&=\begin{cases}
(-1)^m q^{(m-1)(3m-2)} & \text{if $n=2m-1$},\\
(-1)^m q^{m(3m+1)+1} & \text{if $n=2m$}.
\end{cases}\label{eq:th1-1-1}
\end{align}
\end{theorem}

\begin{proof}
It follows from \eqref{eq:pas-var} that
\begin{align*}
&\sum_{r=0}^{n}(-1)^r q^{\binom{r-1}{2}} (-q;q)_{n-r} \qbinom{2n-r+1}{r}\frac{1-q^{2n+1}}{1-q^{2n-r+1}}\\
&\quad = \sum_{r=0}^{n}(-1)^r q^{\binom{r-1}{2}} (-q;q)_{n-r} \left(\qbinom{2n-r}{r-1}+q^r \qbinom{2n-r+1}{r}\right)\\
&\quad = -\sum_{r=1}^{n}(-1)^{r-1} q^{\binom{r-1}{2}} (-q;q)_{n-r}\qbinom{2n-r}{r-1}+q U_n\\
&\quad = -\sum_{r=0}^{n-1}(-1)^r q^{\binom{r}{2}} (-q;q)_{n-r-1}\qbinom{2n-r-1}{r}+q U_n\\
&\quad = -U_{n-1}+q U_n.
\end{align*}
The desired result follows from \eqref{eq:Liu1-1}.
\end{proof}

If we replace $n$ by $2L$, $r$ by $L+k$ and $q$ by $1/q$ in \eqref{eq:th1-1-1}, we obtain another new truncation of Gauss' square exponent theorem

\begin{theorem}
For $L\ge 1$,
\begin{equation}
\sum_{k=-L}^L (-1)^k q^{k^2} (-q;q)_{L-k}\qbinom{3L-k+1}{L+k}\frac{1-q^{4L+1}}{1-q^{3L-k+1}}=1.
\end{equation}
\end{theorem}

We also obtain from \eqref{eq:pas-var}

\begin{theorem}
For $n\ge 1$,
\begin{equation}\label{eq:th-bi}
\sum_{r=0}^{n}(-1)^r q^{\binom{r+1}{2}} (-q;q)_{n-r} \qbinom{2n-r+1}{r}=\sum_{m=-\fl{n/2}}^{\fl{(n+1)/2}}(-1)^m q^{m(3m-1)}.
\end{equation}
\end{theorem}

\begin{proof}
For convenience, we write
$$\tilde{U}_n = \sum_{r=0}^{n}(-1)^r q^{\binom{r+1}{2}} (-q;q)_{n-r} \qbinom{2n-r+1}{r}.$$

It follows from \eqref{eq:pas-var} that
\begin{align*}
&\sum_{r=0}^{n}(-1)^r q^{\binom{r}{2}} (-q;q)_{n-r} \qbinom{2n-r+1}{r}\frac{1-q^{2n+1}}{1-q^{2n-r+1}}\\
&\quad = \sum_{r=0}^{n}(-1)^r q^{\binom{r}{2}} (-q;q)_{n-r} \left(\qbinom{2n-r}{r-1}+q^r \qbinom{2n-r+1}{r}\right)\\
&\quad = -\sum_{r=1}^{n}(-1)^{r-1} q^{\binom{r}{2}} (-q;q)_{n-r}\qbinom{2n-r}{r-1}+\tilde{U}_n\\
&\quad = -\sum_{r=0}^{n-1}(-1)^r q^{\binom{r+1}{2}} (-q;q)_{n-r-1}\qbinom{2n-r-1}{r}+\tilde{U}_n\\
&\quad = -\tilde{U}_{n-1}+\tilde{U}_n.
\end{align*}
From this telescoping identity along with \eqref{eq:th1-1} and the fact that $\tilde{U}_0=1$, we arrive at the desired result.
\end{proof}

\begin{remark}
Letting $n\to \infty$ in \eqref{eq:th-bi} reduces it to
$$(-q;q)_{\infty}\sum_{r=0}^{\infty}\frac{(-1)^r q^{\binom{r+1}{2}}}{(q;q)_r} =\sum_{m=-\infty}^{\infty}(-1)^m q^{m(3m-1)}.$$
We further deduce from Euler's pentagonal number theorem \eqref{eq:eu} that
$$\sum_{r=0}^{\infty}\frac{(-1)^r q^{\binom{r+1}{2}}}{(q;q)_r}=\frac{(q^2;q^2)_\infty}{(-q;q)_\infty}=(q;q)_\infty.$$
This identity, which is a special case of the $q$-binomial theorem (cf.~\cite[Theorem 2.1]{And1976}), is another pioneer work of $q$-identities due to Euler; see \cite[Corollary 2.2]{And1976}. The interested readers may refer to \cite[Theorem 1.1]{Liu2017a} for the following different truncation of this identity
$$\sum_{r=0}^{\fl{n/2}}(-1)^r q^{\binom{r+1}{2}}\qbinom{n-r}{r}=\sum_{m=-\fl{(n+1)/3}}^{\fl{n/3}}(-1)^m q^{m(3m+1)/2}.$$
\end{remark}

\section{Multiple summations}

In \cite{GZ2008}, Guo and Zeng obtained the multiple summation extensions of \eqref{eq:BG} and \eqref{eq:Wa}
\begin{align}
\sum_{j_1,\ldots,j_m=-L}^{2L}\prod_{k=1}^m (-1)^{j_k}& q^{j_k j_{k+1} +\binom{j_k +1}{2}}  \qbinom{2L-j_k}{L+j_{k+1}}\notag\\
&=\begin{cases}
1 & \text{if $m\not\equiv 0 \pmod{3}$},\\
3L+1 & \text{if $m\equiv 0 \pmod{3}$},
\end{cases}\label{eq:GZ-mul-1}\\
\sum_{j_1,\ldots,j_m=-L}^{2L+1}\prod_{k=1}^m (-1)^{j_k}& q^{j_k j_{k+1} +\binom{j_k}{2}} \qbinom{2L-j_k+1}{L+j_{k+1}}\notag\\
&=\begin{cases}
(-1)^{\fl{m^2/3}} & \text{if $m\not\equiv 0 \pmod{3}$},\\
(-1)^{m/3} (3L+2) & \text{if $m\equiv 0 \pmod{3}$}.
\end{cases}\label{eq:GZ-mul-2}
\end{align}
Here we assume that $j_{m+1}=j_1$. The two multiple summations come from a multiple extension of \eqref{eq:BG-1}. Motivated by their work, we study some possible multiple extensions of \eqref{eq:Liu1-1} and \eqref{eq:Liu3-1}.

Parallel to Liu's notation in \cite{Liu2017b}, for positive integer $m$, we put
\begin{align*}
U_m(n)&=\sum_{r_1,\ldots,r_m=0}^{2n+1}\prod_{k=1}^m (-1)^{r_k} q^{\binom{r_k}{2}} (-q;q)_{n-r_k} \qbinom{2n-r_k+1}{r_{k+1}},\\
W_m(n)&=\sum_{r_1,\ldots,r_m=0}^{2n}\prod_{k=1}^m (-1)^{r_k} q^{\binom{r_k}{2}} (-q;q)_{n-r_k} \qbinom{2n-r_k}{r_{k+1}},
\end{align*}
where again we assume that $r_{m+1}=r_1$. Hence $U_1(n)$ and $W_1(n)$ reduce to Liu's $U_n$ and $W_n$, respectively.

We shall show
\begin{theorem}\label{th:multi}
For $n\ge 1$,
\begin{align}
U_2(n)&=0,\\
U_3(n)&=\begin{cases}
0 & \text{if $n=2k-1$},\\
\frac{(-1)^{k-1}}{2} q^{9k^2+3k} & \text{if $n=2k$},
\end{cases}\\
W_2(n)&=(-1)^n q^{n(3n-1)/2},\\
W_3(n)&=\begin{cases}
\frac{(-1)^k}{2}\left(q^{9k^2-9k+3}-3q^{9k^2-11k+3}\right) & \text{if $n=2k-1$},\\
\frac{(-1)^{k-1}}{2} \left(q^{9k^2-3k}-3q^{9k^2-k}\right) & \text{if $n=2k$}.
\end{cases}
\end{align}
\end{theorem}

Instead of using the traditional $q$-series approach, we turn to a computer-assisted proof of Theorem \ref{th:multi}. We recall that Riese implemented a powerful \textit{Mathematica} package \texttt{qMultiSum}, whose main function is generating recurrence relations for multiple summation $q$-identities. We refer to \cite{Rie2003} or the following url
\begin{quote}
\url{http://www.risc.jku.at/research/combinat/software/ergosum/RISC/qMultiSum.html}
\end{quote}
for an introduction to this package.

In our cases, the package gives us
\begin{lemma}\label{th:multi-rec}
For $n\ge 1$,
\begin{align}
0&=-q^{6n+8}U_2(n)+U_2(n+2),\label{eq:U2-rec}\\
0&=q^{9n+12}U_3(n)+U_3(n+2),\\
0&=-q^{9n+11}(1+q^{n+3})W_2(n)-q^{6n+10}(1+q^{n+3})W_2(n+1)\notag\\
&\quad +q^{3n+7}(1+q^{n+1})W_2(n+2)+(1+q^{n+1})W_2(n+3),\\
0&=-q^{15n+24} (1 + q^{n+2}) (-1 + q^{n+3}) (1 + q^{n+3})(1 + q^{n+4}) \notag\\
&\quad\quad\quad\times  \Big(-1 - 
     2 q^{n+2} + q^{n+3}\Big)W_3(n)\notag\\
&\quad -  q^{11 n+23} (1 + q^{n+3}) (1 + q^{n+4}) \notag\\
&\quad\quad\quad\times\Big(-1 + q - 
     2 q^{n+2} - 2 q^{n+3} + 2 q^{n+4} - 6 q^{2n+4} + 
     6 q^{2n+5} - q^{2n+6}\notag\\
&\quad\quad\quad\quad\quad - q^{2n+7} + 2 q^{3n+7} + 
     3 q^{3n+8} - 2 q^{3n+9} - q^{3n+10} + 5 q^{4n+9} - 
     2 q^{4n+10} \notag\\
&\quad\quad\quad\quad\quad- 3 q^{4n+11} + 2 q^{4n+12} - 
     2 q^{5n+12} + 2 q^{5n+13}\Big) W_3(n+1)\notag\\
&\quad +  q^{6n+18} (1 + q^{n+1}) (1 + q^{n+4}) (-1 + q^{2n+5})\notag\\
&\quad\quad\quad\times  \Big(1 + 2 q^{n+1} - q^{n+3} + 7 q^{2n+3} - 
     6 q^{2n+4} - 2 q^{2n+5} + q^{2n+6}  \notag\\
&\quad\quad\quad\quad\quad + 2 q^{3n+6} - q^{3n+8} + q^{4n+10}\Big) W_3(n+2)\notag\\
&\quad +  q^{2n+8} (1 + q^{n+1}) (1 + q^{n+2})\notag\\
&\quad\quad\quad\times  \Big(-2 + 2 q + 
     5 q^{n+2} - 2 q^{n+3} - 3 q^{n+4} + 2 q^{n+5} + 
     2 q^{2n+5} + 3 q^{2n+6}\notag\\
&\quad\quad\quad\quad\quad - 2 q^{2n+7} - q^{2n+8} - 
     6 q^{3n+7} + 6 q^{3n+8} - q^{3n+9} - q^{3n+10} - 
     2 q^{4n+10}\notag\\
&\quad\quad\quad\quad\quad - 2 q^{4n+11} + 2 q^{4n+12} - q^{5n+13} +
      q^{5n+14}\Big) W_3(n+3)\notag\\
&\quad + (1 + q^{n+1}) (-1 + q^{n+2}) (1 + q^{n+2}) (1 + q^{n+3}) \notag\\
&\quad\quad\quad\times  \Big(2 - q + q^{n+3}\Big) W_3(n+4).
\end{align}
\end{lemma}

\begin{proof}
We prove \eqref{eq:U2-rec} by calling (with the initialization \texttt{<<RISC`qMultiSum`})
\begin{lstlisting}[language=Mathematica]
stru = qFindStructureSet[qBinomial[2n-r1+1,r2,q] qBinomial[2n-r2+1,r1,q] (-1)^(r1+r2) q^(r1(r1-1)/2+r2(r2-1)/2) qPochhammer[-q,q,n-r1] qPochhammer[-q,q,n-r2], {n}, {r1,r2}, {1}, {1,1}, {1,1}, qProtocol->True]
rec = qFindRecurrence[qBinomial[2n-r1+1,r2,q] qBinomial[2n-r2+1,r1,q] (-1)^(r1+r2) q^(r1(r1-1)/2+r2(r2-1)/2) qPochhammer[-q,q,n-r1] qPochhammer[-q,q,n-r2], {n}, {r1,r2}, {1}, {1,1}, {1,1}, qProtocol->True, StructSet->stru[[1]]]
sumrec = qSumRecurrence[rec]
\end{lstlisting}
For the remaining three recurrence relations, apart from the corresponding summand, we may set other parameters as follows:
\begin{lstlisting}[language=Mathematica]
{n}, {r1,...,rm}, {1}, {1,...,1}, {1,...,1}
\end{lstlisting}
We remark that it costs over five hours to obtain the recurrence relation for $W_3(n)$.
\end{proof}

\begin{proof}[Proof of Theorem \ref{th:multi}]
Theorem \ref{th:multi} is a direct consequence of Lemma \ref{th:multi-rec} and several initial values.
\end{proof}

Of course, traditional $q$-series proofs of identities in Theorem \ref{th:multi} are cried out. We also notice from Theorem \ref{th:multi} that the multiple extensions of Gauss' square exponent theorem are not as neat as Guo and Zeng's multiple extensions of Euler's pentagonal number theorem (cf.~Corollary 2.3 and Theorem 2.4 of \cite{GZ2008}). However, it would be appealing to see if there exist closed forms of $U_m(n)$ and $W_m(n)$ for arbitrary $m$.

\subsection*{Acknowledgements}

I would like to thank George Andrews for helpful discussions.

\bibliographystyle{amsplain}

\end{document}